\documentclass{amsart}

 \usepackage{amssymb}
\usepackage{amsthm}
\usepackage{texdraw}
\usepackage{subfigure}

\newtheorem{theorem}{Theorem}[section]
\newtheorem{lemma}{Lemma}[section]
\newtheorem{remark}{Remark}[section]
\newtheorem{corollary}{Corollary}[section]

\numberwithin{equation}{section}
\newtheorem*{theorem*}{Theorem}
\newtheorem{definition}{Definition}[section]

\newcommand{\R}{\mathbb{R}}
\newcommand{\Schw}{{\mathcal S}}
\newcommand{\TD}{{\mathcal S'}}
\newcommand{\wh}{\widehat}
\newcommand{\F}{{\mathcal F}}

\newcommand{\be}{\begin{equation}}
\newcommand{\ee}{\end{equation}}

\begin{document}


\title[Weighted convolution inequalities for radial functions]{Weighted convolution inequalities for radial functions}

\author{Pablo L.  De N\'apoli}
\address{IMAS (UBA-CONICET) and Departamento de Matem\'atica, Facultad de Ciencias Exactas y Naturales, Universidad de Buenos Aires, Ciudad Universitaria, 1428 Buenos Aires, Argentina}
\email{pdenapo@dm.uba.ar}

\author{Irene Drelichman}
\address{IMAS (UBA-CONICET) and Departamento de Matem\'atica, Facultad de Ciencias Exactas y Naturales, Universidad de Buenos Aires, Ciudad Universitaria, 1428 Buenos Aires, Argentina}
\email{irene@drelichman.com}

\thanks{Supported by ANPCyT under grant PICT 1675/2010, by CONICET under  grant PIP 1420090100230 and by Universidad de
Buenos Aires under grant 20020090100067. The authors are members of
CONICET, Argentina.}

\begin{abstract}
We obtain convolution inequalities in Lebesgue and Lorentz spaces with power weights when the functions involved are assumed to be radially symmetric. We also present applications of these results to  inequalities for Riesz potentials of radial functions in weighted Lorentz spaces and embedding theorems for radial Besov spaces with power weights. 

\end{abstract}

\maketitle

\section{Introduction}
The aim of this paper is to study boundedness properties of  the convolution operator
$$
(f*g)(x) = \int_{\mathbb{R}^n} f(x-y)g(y) \, dy
$$ 
in Lebesgue and Lorentz spaces with power weights, when restricted to radially symmetric functions. 

To state our results, first we need to introduce some notation. Given a measurable function $f$ in $\mathbb{R}^n$, we denote its distribution function with respect to the weight $w(x)=|x|^{\alpha p}$ by
$$
\mu_f (s) = \int_{\{ x : |f(x)|>s\}} |x|^{\alpha p} \, dx , \quad s>0.
$$

The weighted Lorentz space $L(p,q; \alpha)$ is the space of all measurable functions in $\mathbb{R}^n$ such that $\|f\|_{p,q;\alpha}$ is finite, with
$$
\|f\|_{p,q;\alpha} = \left( q \int_0^\infty s^{q-1} \mu_f(s)^{\frac{q}{p}} \, ds \right)^{\frac{1}{q}}, \quad  1<p<\infty, 1\le q<\infty, 
$$
$$
\|f\|_{p,\infty; \alpha} = \sup_{s>0} s \mu_f(s)^{\frac{1}{p}} , \quad  1\le p < \infty.
$$

When $p=q$, we recover the weighted Lebesgue space $L(p,p;\alpha)=L(p;\alpha)$ with
$$
\|f\|_{p,\alpha} = \left( \int_{\mathbb{R}^n} |f(x)|^p |x|^{\alpha p} \, dx \right)^{\frac{1}{p}}, \quad1\le p<\infty 
$$
$$
\|f\|_{\infty;Ê\alpha} ={ \mathrm{ess} \sup}_{x\in \mathbb{R}^n} |f(x) w(x)|.
$$
When $\alpha=0$ we simply write $L^p$. Finally, by $L_{rad}(p,q;\alpha)$,  $L_{rad}(p, \alpha)$ or $L^p_{rad}$ we denote the subspaces of radial functions of the corresponding spaces. 

Following \cite{K}, given functional spaces $X, Y, Z$, we shall write $X * Y \subset Z$ to indicate that for functions $f \in X, g\in Y$, then $f*g \in Z$ and there exists a positive constant $C$ such that
$$
\|f*g\|_{Z} \le C \|f\|_X \|g\|_Y.
$$
Then, the classical Young's inequality reads 

\begin{theorem}(Young's inequality)
$$L_p * L_q \subset L_r$$
for $1\le p,q,r \le \infty$, provided that $\frac{1}{r} = \frac{1}{p}+ \frac{1}{q} -1.$ \end{theorem}

In Lorentz spaces the result is due to O'Neil \cite{O}:
\begin{theorem}
$$L(p_0,q_0) * L(p_1,q_1) \subset L(p,q)$$
for $1< p_0, p_1, p< \infty$, provided that $\frac{1}{p} = \frac{1}{p_0} + \frac{1}{p_1} -1$ and $0\le \frac{1}{q} \le \frac{1}{q_0} + \frac{1}{q_1} \le 1$. \end{theorem}

 In the case  of power weights, the above theorems were generalized by R. Kerman \cite{K}. Partial results for the $L^p$ case can also be found in \cite{B}.
\begin{theorem}\cite[Theorem 3.1]{K}
\label{teo-kerman}
$$
L (p ; \alpha) * L (q; \beta) \subset  L (r; - \gamma)
$$
provided
\begin{enumerate}
\item $\frac{1}{r} = \frac{1}{p} + \frac{1}{q} + \frac{\alpha + \beta + \gamma}{n} -1, 1<p,q,r<\infty  , \frac{1}{r} \le \frac{1}{p} + \frac{1}{q},$
\item $ \alpha<\frac{n}{p'} , \beta< \frac{n}{q'}  , \gamma < \frac{n}{r},$
\item $\alpha + \beta \ge0, \beta + \gamma \ge0 , \gamma + \alpha \ge 0.$
\end{enumerate}
\end{theorem}

\begin{theorem}\cite[Theorem 4.1]{K}
\label{teo-kerman-lorentz}
$$
L (p_0, q_0 ; \alpha) * L (p_1, q_1 ; \beta) \subset  L (p, q ; - \gamma)
$$
provided
\begin{enumerate}
\item $\frac{1}{p}=\frac{1}{p_0}+\frac{1}{p_1} + \frac{\alpha+\beta+\gamma}{n}-1,  <p_0, p_1, p< \infty, \le \frac{1}{q}\le \frac{1}{q_0}+\frac{1}{q_1}\le 1,$
\item $\alpha<\frac{n}{p_0'} ,\beta< \frac{n}{p_1'} ,\gamma < \frac{n}{p},$
\item $\alpha + \beta  >0,  \beta + \gamma >0 , \gamma + \alpha > 0.$
\end{enumerate}
\end{theorem}

Further weighted inequalities for convolutions  can be found in \cite{BS}, \cite{KS}, \cite{NT}, \cite{Ra}  (see also references therein). However, the fact that one can improve Theorems \ref{teo-kerman} and \ref{teo-kerman-lorentz} when the functions involved are assumed to be radial was seemingly overlooked, and is the object of the present paper. Namely, we will prove:

\begin{theorem}
\label{teo-radiales}
$$
L_{rad} (p ; \alpha) * L_{rad} (q; \beta) \subset  L_{rad} (r; - \gamma)
$$
provided
\begin{enumerate}
\item $\frac{1}{r} = \frac{1}{p} + \frac{1}{q}+ \frac{\alpha + \beta + \gamma}{n} -1 , 1<p,q,r<\infty ,  \frac{1}{r} \le \frac{1}{p} + \frac{1}{q},$
\item $\alpha<\frac{n}{p'} , \beta< \frac{n}{q'} , \gamma < \frac{n}{r},$
\item $\alpha + \beta \ge {(n-1)(1 - \frac{1}{p} - \frac{1}{q})} , \beta + \gamma \ge (n-1)(\frac{1}{r}- \frac{1}{q}), \gamma + \alpha \ge (n-1)(\frac{1}{r} - \frac{1}{p})$
\item  $\max\{\alpha, \beta , \gamma\} > 0$ or $\alpha=\beta=\gamma=0$.
\end{enumerate}
\end{theorem}

\begin{theorem}
\label{teo-lorentz}
$$
L_{rad} (p_0, q_0 ; \alpha) * L_{rad} (p_1, q_1 ; \beta) \subset  L_{rad} (p, q ; - \gamma)
$$
provided
\begin{enumerate}
\item $\frac{1}{p}=\frac{1}{p_0}+\frac{1}{p_1} + \frac{\alpha+\beta+\gamma}{n}-1$ , $1<p_0, p_1, p< \infty$, $0\le \frac{1}{q}\le \frac{1}{q_0}+\frac{1}{q_1}\le 1$,
\item $\alpha<\frac{n}{p_0'}$ , $\beta< \frac{n}{p_1'}$ , $\gamma < \frac{n}{p},$
\item $\alpha + \beta > {(n-1)(1 - \frac{1}{p_0} - \frac{1}{p_1})} , \beta + \gamma > (n-1)(\frac{1}{p}- \frac{1}{p_1}), \gamma + \alpha > (n-1)(\frac{1}{p} - \frac{1}{p_0}),$
\item $\max\{\alpha, \beta , \gamma\} > 0.$
\end{enumerate}
\end{theorem}

\begin{remark}
\label{pigual1}
Theorem \ref{teo-radiales} also holds for $r=1$ and will be proved separately (see Theorem \ref{rigual1}). Moreover, it can be seen from the proof of Theorem \ref{teo-radiales}, that it also holds for  $p=1$ (or $q=1$, but not both), provided the inequalities in $\alpha+\beta$, $\beta+\gamma$ and $\alpha+\gamma$ are strict and $\beta <n(\frac{1}{r}-\frac{1}{p})$  (respectively, $\alpha <n(\frac{1}{r}-\frac{1}{p})$). 
\end{remark}

Notice that the hypotheses of Theorem \ref{teo-kerman} are more restrictive than those of Theorem \ref{teo-radiales} (see Figure 1 below for a comparison in a special case). Indeed, in the first case, at most one among $\alpha, \beta, \gamma$ can be negative, while in the latter this condition is relaxed to at most two. Moreover, the condition on $\alpha+\beta$ in Theorem  \ref{teo-radiales}  is only seemingly more restrictive than that of Theorem \ref{teo-kerman} when $1-\frac{1}{p}-\frac{1}{q} > 0$, since in that case one has that
$$
\alpha+\beta =  n \left(1 - \frac{1}{p} - \frac{1}{q} + \frac{1}{r} - \frac{\gamma}{n} \right) >  n \left( 1 - \frac{1}{p} - \frac{1}{q} \right) >  (n-1) \left( 1 - \frac{1}{p} - \frac{1}{q} \right).
$$

Similar considerations apply also to the conditions on $\alpha + \gamma$ and $\beta + \gamma$, and in the case of Theorem \ref{teo-lorentz} compared to Theorem \ref{teo-kerman-lorentz}. 

\begin{figure}[!ht]
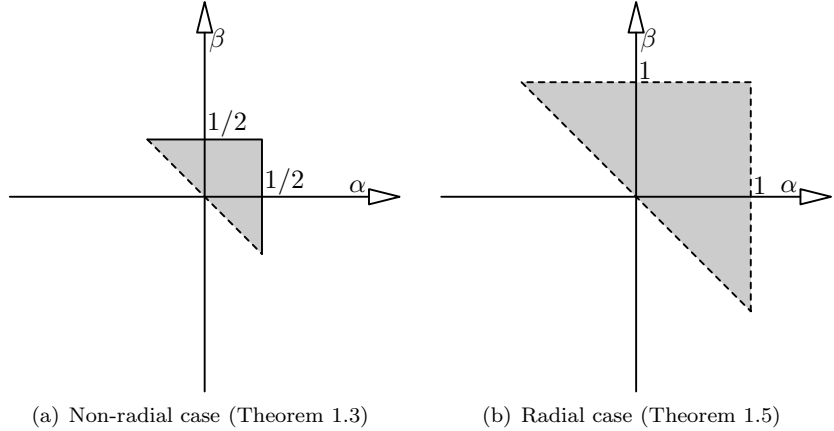

\subfigure[Non-radial case (Theorem \ref{teo-kerman})]{
\begin{texdraw}
\setunitscale 0.6
\lpatt ()
\move (0.5 -0.5)
\lvec (0.5 0.5)
\lvec (-0.5 0.5)
\lvec (0.5 -0.5)
\ifill f:0.8
\lpatt ()
\move (0.5 -0.5)
\lvec (0.5 0.5)
\lvec (-0.5 0.5)
\lpatt(0.05 0.05)
\lvec (0.5 -0.5)
\lpatt()
\move (0 -1.7)
\avec (0 1.7)
\htext (0.03 1.25){$\beta$}
\move (-1.7 0)
\avec (1.7 0)
\htext (1.25 0.03){$\alpha$}
\htext (0.51 0.01){1/2}
\htext (0.01 0.51){1/2}
\end{texdraw}}
\quad
\subfigure[Radial case (Theorem \ref{teo-radiales})]{
\begin{texdraw}
\setunitscale 0.6
\lpatt(0.05 0.05)
\move (1 -1)
\lvec (1 1)
\lvec (-1 1)
\lvec (1 -1)
\lfill f:0.8
\lpatt()
\move (0 -1.7)
\avec (0 1.7)
\htext (0.03 1.25){$\beta$}
\move (-1.7 0)
\avec (1.7 0)
\htext (1.25 0.03){$\alpha$}
\htext (1.01 0.01){1}
\htext (0.01 1.01){1}
\end{texdraw}}

\caption{Comparison of admissible values in the $\alpha \beta$-plane when $n=2$, $p=q=2$, $r=4$}
\end{figure}

The restriction to radially symmetric functions is natural, for instance, for applications to partial differential equations in $\mathbb{R}^n$, but one could also be tempted to ask whether one can improve the conditions on $\alpha+\beta$, $\alpha+\gamma$ and $\beta+\gamma$ of Theorem \ref{teo-kerman} by restricting the convolution to functions invariant with respect to a different subgroup of the orthogonal group. However, the answer is negative, as we will show below (see Remark \ref{subgrupos}).

The rest of the paper is as follows. In Section 2 we prove some preliminary results; in Section 3 we prove Theorem \ref{teo-radiales} and in Section 4 we outline the proof of Theorem \ref{teo-lorentz} and  obtain, as a corollary, weighted estimates for fractional integrals of radial functions in Lorentz spaces. Finally, in Section 5 we present embedding theorems with power weights for radial Besov spaces.

\section{Preliminary results}

First  we show that the conditions of Theorem \ref{teo-kerman} on $\alpha+\beta$, $\beta+\gamma$ and $\alpha+\gamma$ cannot be improved for arbitrary functions, or indeed for any set of functions invariant with respect to a subgroup of the orthogonal group other than the radial functions. This can be done by applying the convolution inequality to the heat kernel, as was done in \cite{B} to prove the necessity of the scaling and integrability conditions (see \cite[Theorem 2.1]{B}).
 
 \begin{remark}
 If
$$ 
L (p ; \alpha) * L (q; \beta) \subset  L (r; - \gamma) $$
 then  $\alpha+\beta\ge 0$, $\alpha+\gamma\ge 0$ and $\beta+\gamma\ge 0$. 
 
 \end{remark}
 \begin{proof}
 One can verify easily  that $L (p ; \alpha) * L (q; \beta) \subset  L (r; - \gamma)$  implies the scaling condition
$$
 \frac{1}{r}=\frac{1}{p}+\frac{1}{q}+\frac{\alpha+\beta+\gamma}{n}-1 \quad 
$$
 (or see \cite{B} for proof). We will first show that $\alpha+\beta\geq 0$.

\medskip

Let $W_t(x)= (4\pi t)^{-n/2} e^{-|x|^2/(4t)}$ be the heat kernel in $\R^n$ and let
$W_{t,y}(x) = W_t(x-y) \; \hbox{for} \; y \in \R^n$. Then, we have $ W_{t,y} * W_{t,-y} = W_{2t}$ whence, if the claimed inclusion
holds, one gets
\begin{equation}
\label{desig2}
\|W_{2t}\|_{r;-\gamma} \le C \|W_{t,y}\|_{p;\alpha} \|W_{t,-y}\|_{q;\beta} \quad  \forall \; y \in \R^n, \forall \; t>0.
\end{equation}

Now, setting $z=\frac{x}{\sqrt{2t}}$, we have that
$$
\|W_{2t}\|_{r;-\gamma} = (2t)^{(-n-\gamma +n/r)/2} \|W_1\|_{r;-\gamma} = t^{(-n-\gamma +n/r)/2} C(n,\gamma,r).
$$

Similarly, setting $z=\frac{x-y}{\sqrt{t}}$, 
$$
\|W_{t,y}\|_{p;\alpha}=  t^{(-n+\alpha+n/p)/2} 
\left( \int_{\R^n} \left|W_1(z)\right|^{p} \, \, \Big|z+\frac{y}{\sqrt{t}}\Big|^{\alpha p} \; dz  \right)^{1/p},$$
and setting $z=\frac{x+y}{\sqrt{t}}$,
$$
\|W_{t,-y}\|_{q;\beta} =  t^{(-n+\beta+n/q)/2} \left( \int_{\R^n} |W_1(z)|^{q} \left|z-\frac{y}{\sqrt{t}}\right|^{\beta q} \; dz  \right)^{1/q}.$$

Therefore, replacing in (\ref{desig2}) and noting that the powers of $t$  cancel out due to the scaling condition , one has that, for some $C>0$ depending only on the parameters $p,q,r,\alpha,\beta,\gamma$ and the dimension $n$, 
$$
 C  \leq  
\left( \int_{\R^n} |W_1(z)|^{p} \left|z+\frac{y}{\sqrt{t}}\right|^{\alpha p} \; dz  \right)^{1/p} 
\left( \int_{\R^n} |W_1(z)|^{q} \left|z-\frac{y}{\sqrt{t}}\right|^{\beta q} \; dz  \right)^{1/q}.
$$

Choosing $y= \lambda \sqrt{t} \; y_0$ for some $y_0 \in \R^n$ fixed with $|y_0|=1$, we get, for any $\lambda>0$,
$$
 C  \lambda^{-\alpha-\beta} \leq  
\left( \int_{\R^n} |W_1(z)|^{p} \left|\frac{z}{\lambda} + y_0\right|^{\alpha p} \; dz  \right)^{1/p} 
\left( \int_{\R^n} |W_1(z)|^{q} \left|\frac{z}{\lambda}-y_0 \right|^{\beta q} \; dz  \right)^{1/q}
$$
whence, letting $\lambda \to +\infty$, we deduce that $\alpha+\beta \geq 0$, since otherwise we would get 
$\|W_1\|_p \|W_1\|_q\ge +\infty$,
a contradiction.

In a similar way, using the relations
$$ W_{t} * W_{t,y} = W_{2t,y} $$
and
$$ W_{t,y} *  W_{t} = W_{2t,y} $$
(or a duality argument) we can prove the necessity of the conditions $\beta+\gamma\geq 0$, $\alpha+\gamma\geq 0$.
\end{proof}

\begin{remark}
 \label{subgrupos}

A natural question is whether the conditions $\alpha+\beta\ge 0$, $\alpha+\gamma\ge 0$ and $\beta+\gamma\ge 0$ can be improved, if we restrict our atention
to functions invariant by some subgroup $G$ of the orthogonal group $O(n)$, for instance if these conditions can be relaxed for functions in $\R^3$ with cylindrical symetry, which are invariant with respect to rotations in the $x,y$-plane (a subgroup of $O(3)$ isomorphic to $O(2)$). 

\medskip

The argument in the preceeding remark also shows that they cannot be improved if the action of $G$ on the sphere $S^{n-1}$ has a fixed point, 
i.e. if there exist some $y_0$, $|y_0|=1$, such that 
 $$ g \cdot y_0 = y_0 \qquad \forall \; g \in G.$$
(as in the case of functions with cylindrical symmetry in $\R^3$). It is enough to observe that if we choose $y_0$ in the previous argument as that fixed point, then 
the functions $W_{t,y}$, $W_{t,-y}$ and $W_{2t}$, with $y=\lambda \sqrt{t} \; y_0$ as before are $G$-invariant, so the same argument applies.
\end{remark}

Now, note that a key point in the proof of Theorems \ref{teo-kerman} and \ref{teo-kerman-lorentz} is the relationship between the convolution operator and the fractional integral, or Riesz potential, given by
\begin{equation}
\label{int-frac}
(T_\gamma f)(x)= \int_{\R^n} \frac{f(y)}{|x-y|^\gamma} \; dy, \quad
0 <\gamma < n. 
\end{equation}
Indeed, the proof in \cite{K}  invokes known weighted estimates for this operator proved by E. Stein and G. Weiss \cite{SW}. We will follow that method of proof but use instead the following result proved by the authors and R.G. Dur\'an in \cite[Theorem 1.2]{DDD}, that  gives the weighted inequalities when $T_\gamma$ is restricted to radial functions. The result for $p>1$ was previously proved by different means in \cite[Theorem 1.2]{R}. An alternative proof for all $p\ge 1$ can also be found in \cite[Theorem 5.1]{D}, where also sharpness of the result  is proved. 
\begin{theorem}\cite[Theorem 1.2]{DDD}
\label{teo-ddd}
Let $n\ge 1$, $0<\gamma<n, 1<p<\infty, \alpha<\frac{n}{p'},
\beta<\frac{n}{q}, \alpha + \beta \ge (n-1)(\frac{1}{q}-\frac{1}{p})$, and
$\frac{1}{q}=\frac{1}{p}+\frac{\gamma+\alpha+\beta}{n}-1$.  If $p\le
q<\infty$, then the inequality
$$
\|T_\gamma f\|_{q;-\beta} \le C \|f\|_{p;\alpha}
$$
holds for all radially symmetric $f\in L^p(\mathbb{R}^n, |x|^{p\alpha} dx)$,
where $C$ is independent of $f$. If $p=1$, then the result holds provided $\alpha +
\beta > (n-1)(\frac{1}{q}-1)$.
\end{theorem}

Once we establish Theorem \ref{teo-lorentz}, we will also be able to extend this result to Lorentz spaces. However, for now we postpone a proof of this fact
and prove that, as an almost immediate consequence of the previous theorem, one has the following special case of our convolution inequalities, that we will need later:

\begin{theorem}
\label{rigual1}
If, for $1<p,q< \infty$, we have
$$
2= \frac{1}{p} + \frac{1}{q} + \frac{\alpha+\beta+\gamma}{n} \quad , \quad \frac{1}{p} + \frac{1}{q} \ge 1 \quad , 
$$
$$
\alpha < \frac{n}{p'} \quad , \quad \beta < \frac{n}{q'} \quad , \quad 0<\gamma < n  \quad , 
$$
and 
\begin{equation}
\label{condicion-alfabeta}
\alpha +\beta \ge (n-1) \left (1 - \frac{1}{p} - \frac{1}{q}\right),
\end{equation}
then
$$
L_{rad}(p; \alpha) * L_{rad}(q; \beta) \subset L_{rad}(1; -\gamma).
$$
The result also holds for $p=1$ (or $q=1$) provided that the inequality in (\ref{condicion-alfabeta}) is strict. 
\end{theorem} 
\begin{proof}
It suffices to consider the case $f, g \ge 0$. Then, by Tonelli's theorem,
$$
\int_{\mathbb{R}^n} (f*g)(x) |x|^{-\gamma} \, dx = \int_{\mathbb{R}^n} (T_\gamma f)(x) g(-x) \, dx
$$
whence, by H\"older's inequality and Theorem \ref{teo-ddd}, given our conditions, we obtain
$$
\| f*g\|_{1; -\gamma} \le C \|T_\gamma f\|_{q'; -\beta} \|g\|_{q; \beta} \le C \|f\|_{p; \alpha} \|g\|_{q; \beta}
$$
\end{proof}

The above result coincides with Theorem \ref{teo-radiales} for $r=1$. The result for other values of $r$ will be obtained by duality and multilinear interpolation. The next two results make this explicit. 
\begin{lemma}
\label{lema-dualidad}
Suppose $a$ and $b$ are real numbers and $1<r\le \infty$, $1\le s < \infty$. Let $f, g$ be nonnegative functions on $\mathbb{R}^n$, $f$ radially symmetric, and define the linear operator $T_f$ by
$$
(T_f g)(x) = \int_{\mathbb{R}^n} f(x-y) g(y) \, dy.
$$
Then,
$$
T_f: L_{rad}(r; a) \to L_{rad}(s; b),
$$
with operator norm $C$ is equivalent to
$$
T_f : L_{rad}(s'; -b) \to L_{rad}(r'; -a)
$$
with the same norm. 
\end{lemma}
\begin{proof}
It follows easily by using duality. Or follow the proof in \cite[Lemma 3.2]{K} restricting the operator to radial functions.
\end{proof}

\begin{theorem}[\cite{C}]
\label{interpolacion}
Suppose $T$ is a multilinear operator satisfying 
$$
T:  L(p_i, w_i) \times L(p_i', w_i') \to L(p_i'', w_i'')
$$
with norm $K_i, i=0,1$. Then, 
$$
T: L(p_t, w_t) \times L(p_t', w_t') \to L(p_t'', w_t'')
$$
with norm at most $K_0^{1-t} K_1^t,$ where $p_t, p_t', p_t'', w_t, w_t'$ and $w_t''$ are given by 
$$
\frac{1}{p_t}= \frac{1-t}{p_0} +\frac{t}{p_1} \quad , \quad \frac{1}{q_t} = \frac{1-t}{q_0} + \frac{t}{q_1} \quad \mbox{and} \quad w_t = w_0^{p_t(1-t)/p_0}w_1^{p_t t/p_1}. 
$$
\end{theorem}

\begin{remark}
When $w_i= |x|^{\alpha_i p_i}, w_i'= |x|^{\alpha_i' p_i'}, w_i''= |x|^{\alpha_i'' p_i''}$, clearly one has $\alpha_t = (1-t) \alpha_0 + t \alpha_1$, $\alpha_t' = (1-t)\alpha_0' + t \alpha_1'$ and $\alpha_t'' = (1-t)\alpha_0'' + t \alpha_1''.$
\end{remark}

\begin{remark}
\label{interpolacion-rad}
Since we will actually use the above theorem to interpolate between the subspaces of radial functions of the corresponding spaces, a comment is in order.  In general, one cannot freely interpolate between subspaces and guarantee that the intermediate space is the expected subspace. However, the subspaces $L_{rad}(p,\alpha)$ in $\mathbb{R}^n$ are isomorphic to spaces $L(p, \alpha + \frac{n-1}{p})$ in $(0,\infty)$, so one may interpolate in the latter setting and use the fact that the interpolation commutes with the standard isomorphism.
\end{remark}

\section{Convolution in weighted Lebesgue spaces }

In this section we prove Theorem \ref{teo-radiales}.  First we claim the following:

\begin{remark} \label{reduccion}
By Theorem  \ref{teo-kerman}, it suffices to establish Theorem  \ref{teo-radiales} in the following cases:
\begin{enumerate}
\item $\alpha + \beta \le 0, \gamma > 0$ \label{caso1}
\item $\beta + \gamma \le 0 , \alpha > 0$ \label{caso2}
\item $\alpha + \gamma \le 0 , \beta > 0$ \label{caso3}
\end{enumerate}
\end{remark}
\begin{proof}

We may assume $\max \{\alpha, \beta, \gamma\}=\gamma >0$ since the other cases follow similarly. Moreover, by symmetry between $\alpha$ and $\beta$, it suffices to consider the following three possibilities: 

 $\gamma > 0 \ge \alpha \ge \beta$. This is contained in case  \ref{caso1} above.
 
  $\gamma \ge \alpha  \ge \beta \ge 0$, $\gamma \neq 0$. This is contained in Theorem \ref{teo-kerman}.

  $\gamma \ge \alpha  \ge 0 \ge  \beta$, $\gamma \neq 0$. If $\alpha+\beta \le 0$, it reduces to case  \ref{caso1} above. If instead $\alpha+ \beta >0$, it  follows that $\alpha+\gamma\ge 0$, $\beta+\gamma \ge 0$ and the result is again contained in Theorem \ref{teo-kerman}.

\end{proof}

\begin{proof}(\emph{Theorem \ref{teo-radiales}})  We begin by considering the case $\alpha+\beta \le 0, \gamma > 0$.

Now, $\alpha+\beta \le 0$ clearly implies $\frac{1}{p}+\frac{1}{q} \ge 1$, so this case is analogous to the first case of the proof of Theorem \ref{teo-kerman} given in \cite{K}, and
the result follows by interpolation between H\"older's inequality and  Theorem \ref{rigual1}, that is, using Theorem \ref{interpolacion} with the endpoints 
$$
r_0 = \infty \quad , \quad \frac{1}{p_0}+\frac{1}{q_0} = 1 \quad , \quad \alpha_0=\beta_0 =\gamma_0 = 0
$$
and
$$
r_1 = 1 \quad , \quad  2 = \frac{1}{p_1} + \frac{1}{q_1} + \frac{\alpha_1 + \beta_1 + \gamma_1}{n} \quad , \quad  \frac{1}{p_1}+\frac{1}{q_1} \ge 1
$$
$$
 \alpha_1 < \frac{n}{p_1'} \quad , \quad \beta_1 < \frac{n}{q_1'}  \quad , \quad 0< \gamma_1 <n 
$$
$$
\alpha_1+\beta_1 \ge (n-1)\left(1 - \frac{1}{p_1} - \frac{1}{q_1}\right).
$$
Hence, $t=\frac{1}{r}$ and $\alpha_1= r\alpha \, , \,  \beta_1 = r \beta \, , \,  \gamma_1 = r \gamma$. 

We need to check that the hypotheses of Theorem \ref{rigual1} are satisfied. Clearly $0< \gamma_1 < n$. The remaining conditions depend on the choice of $p_0$. We begin by considering $\beta_1 < \frac{n}{q_1'}$, which is equivalent to 
\begin{equation}
\label{cond-p0}
\frac{1}{p_0} < \frac{1-\frac{1}{q} - \frac{\beta}{n}}{1-\frac{1}{r}}.
\end{equation}

In order to choose $p_0$ we consider the following two cases: 

\emph{Case 1:}  If the right hand side is greater than 1, that is $\frac{1}{r} - \frac{1}{q} > \frac{\beta}{n}$, we choose $p_0=1$. Then,  $q_0 = \infty \, , \, \frac{1}{p_1} = r (\frac{1}{p} + \frac{1}{r} - 1)$ and $\frac{1}{q_1} = \frac{r}{q}$. Given this, one can check that condition $2= \frac{1}{p_1} + \frac{1}{q_1} + \frac{\alpha_1 + \beta_1 + \gamma_1}{n}$ follows from $\frac{1}{r}= \frac{1}{p}+\frac{1}{q} + \frac{\alpha + \beta + \gamma}{n} -1$, that condition $\alpha_1<\frac{n}{p_1'}$ follows from $\alpha<\frac{n}{p'}$, that condition $\frac{1}{p_1} + \frac{1}{q_1} \ge 1$ follows from $\frac{1}{p} + \frac{1}{q} \ge 1$, and that condition $\alpha_1 + \beta_1 \ge (n-1)(1 - \frac{1}{p_1} - \frac{1}{q_1})$ follows from $\alpha+\beta \ge (n-1)(1-\frac{1}{p} - \frac{1}{q}).$

\emph{Case 2:} If $\frac{1}{r} - \frac{1}{q} \le \frac{\beta}{n}$, we choose $\frac{1}{p_0}=(1-\varepsilon)(1-\frac{1}{q} - \frac{\beta}{n})/(1-\frac{1}{r})$ for small positive $\varepsilon$ that we will choose later. Then,  $\frac{1}{q_0} =  [\frac{1}{q}+\frac{\beta}{n}-\frac{1}{r}+ \varepsilon (1-\frac{1}{q}-\frac{\beta}{n})]/(1-\frac{1}{r})$, $\frac{1}{p_1} = r[\frac{1}{p} - (1-\varepsilon)(1 + \frac{1}{q} + \frac{\beta }{n})]$ and $\frac{1}{q_1} = r[\frac{1}{r} - \frac{\beta}{n}-\varepsilon(1-\frac{1}{q}-\frac{\beta}{n})].$ Therefore, condition $\alpha_1 < \frac{n}{p_1'} $ follows from the scaling condition and the fact that $\gamma>0$, provided we choose $\varepsilon<\gamma/(\frac{n}{q'}-\beta)$, condition $\frac{1}{p_1} + \frac{1}{q_1} \ge 1$ follows from $\frac{1}{p} + \frac{1}{q} \ge 1$, condition $2=\frac{1}{p_1} + \frac{1}{q_1} + \frac{\alpha_1 + \beta_1 + \gamma_1}{n}$ follows from $\frac{1}{r} = \frac{1}{p} + \frac{1}{q} + \frac{\alpha+\beta+\gamma}{n} -1$ and condition $\alpha_1 + \beta_1  \ge (n-1)(1-\frac{1}{p_1} - \frac{1}{q_1})$ follows from $\alpha + \beta \ge (n-1)(1-\frac{1}{p} - \frac{1}{q})$.

This completes the case $\alpha + \beta \le 0 \, , \, \gamma> 0$. By Lemma \ref{lema-dualidad}, one has then the result for $\gamma + \beta \le 0  \, , \, \alpha > 0$ and $\alpha + \gamma \le 0 \, , \, \beta > 0$, which in view of Remark \ref{reduccion} complete the proof.
\end{proof}

\begin{remark}
If $p=1$, as mentioned in Remark \ref{pigual1}, one can choose $p_0=p_1=1$ above, since obviously $\frac{1}{p}+\frac{1}{q}\ge 1$. However, by equation (\ref{cond-p0}), in this case one needs $\beta<n(\frac{1}{r}-\frac{1}{q})$ to follow through the proof. By symmetry, if $q=1$, one needs $\alpha<n(\frac{1}{r}-\frac{1}{p})$.
\end{remark}

\section{Convolution in weighted Lorentz spaces}

\begin{proof}\emph{(Theorem \ref{teo-lorentz})}. The proof can be carried out exactly as the proof of Theorem \ref{teo-kerman-lorentz} given in \cite{K}, once we establish that,  under the assumptions of the theorem, the restricted weak type inequality
$$
\int_H (\chi_F * \chi_G) (x)  |x|^{-\gamma p} \, dx 
$$
$$
\le  C\left(\int_F |x|^{\alpha p_0} \, dx \right)^{\frac{1}{p_0}} \left(\int_G |x|^{\beta p_1} \, dx \right)^{\frac{1}{p_1}} \left(\int_H |x|^{\gamma p} \, dx \right)^{\frac{1}{p'}}
$$
holds  for $F,G,H \subset \mathbb{R}^n$ of finite measure, such that $\chi_F$ and $\chi_G$ are radial. Indeed, 
if $\frac{1}{p} \le \frac{1}{p_0} + \frac{1}{p_1}$, by H\"older's inequality 
$$
\int_{H} (\chi_F * \chi_G )(x) |x|^{-\gamma p} \, dx \le C \|\chi_F * \chi_G\|_{p; -\gamma} \left(\int_H |x|^{-\gamma p} \, dx \right)^{\frac{1}{p'}}
$$
and the result follows by Theorem \ref{teo-radiales}. 
If instead one has $\frac{1}{p}>\frac{1}{p_0} + \frac{1}{p_1}$, the result holds even for non necessarily radial functions and is contained in \cite[Proposition 4.2]{K}. 

Since the rest of the proof is as in \cite{K}, noticing again that one may interpolate between radial subspaces for similar reasons as those of Remark \ref{interpolacion-rad}, we leave the details to the reader. 
\end{proof}

As an application of Theorem \ref{teo-lorentz} one has the following result for Riesz potentials (defined by (\ref{int-frac})) of radial functions in Lorentz spaces with power weights, that extends the result obtained in \cite[Theorem 4.5]{K} for non necessarily radial functions in a similar way.

\begin{theorem}
Let $0<\lambda<n$, $1<p_0<\infty$, $\alpha<\frac{n}{p_0'}$, $\gamma<\frac{n}{p}$, $\alpha+\gamma> (n-1)(\frac{1}{p}-\frac{1}{p_0})$ and $\frac{1}{p}=\frac{1}{p_0} + \frac{\alpha+\lambda+\gamma}{n}-1$. Then, 
$$
T_\lambda : L_{rad}(p_0, q_0; \alpha) \to L_{rad}(p, q; -\gamma)
$$
for $q\ge q_0.$
\end{theorem}
\begin{proof}
Assume $f$ radial. Then, by Theorem \ref{teo-lorentz},
$$
\|T_\lambda f\|_{p,q; -\gamma} \le C \| |x|^{-\lambda}\|_{p_1,q_1;\beta} \, \|f\|_{p_0, q_0; \alpha}
$$
provided we also have $\frac{1}{p}=\frac{1}{p_0}+\frac{1}{p_1}+\frac{\alpha+\beta+\gamma}{n}-1$, $\beta<\frac{n}{p_1'}$, $\alpha+\beta>(n-1)(1-\frac{1}{p_0}-\frac{1}{p_1})$, $\beta+\gamma>(n-1)(\frac{1}{p}-\frac{1}{p_1})$ and $\max\{\alpha, \beta, \gamma\}>0.$ 

We claim that it suffices to take $q_1=\infty$, $\frac{1}{p_1} < \min\{\frac{1}{p_0'}, \frac{1}{p}\}$ and let $\frac{\beta}{n}=\frac{\lambda}{n}-\frac{1}{p_1}.$ Indeed, with this choice of parameters, $ \| |x|^{-\lambda}\|_{p_1,q_1;\beta} <\infty$ and $\frac{1}{p}=\frac{1}{p_0} + \frac{\alpha+\lambda+\gamma}{n}-1$.
Moreover, by the conditions on $\alpha$ and $\gamma$ and the choice of $p_1$,
$$
\alpha+\beta = n\left( \frac{1}{p}- \frac{1}{p_0} -\frac{1}{p_1} -\frac{\gamma}{n} + 1\right) > (n-1)\left(1-\frac{1}{p_0}-\frac{1}{p_1}\right)
$$
and
$$
\beta+\gamma = n\left( \frac{1}{p}- \frac{1}{p_0} -\frac{1}{p_1} -\frac{\alpha}{n} + 1\right) > (n-1)\left(\frac{1}{p}-\frac{1}{p_1}\right).
$$
Finally, since by the condition on $p_1$
$$
\frac{\alpha+\beta+\gamma}{n} = \frac{1}{p} + \frac{1}{p_0'}- \frac{1}{p_1} > 0,
$$
one clearly has $\max\{\alpha,\beta, \gamma\}>0,$ as required. 

\end{proof}

\section{Applications to weighted embeddings of radial Besov spaces}

In this section we show how Theorem \ref{teo-radiales} can be used to obtain weighted embedding theorems for radial Besov spaces with power weights. We will follow closely the proof recently given by Meyries and Veraar  in the non-radial case \cite{MV} (see also references therein for previously known results); this will allow us to be sketchy in the standard part of the proof and to point out precisely in which steps one can obtain improvements in the radial case.

 One should note that the result in  \cite{MV} is not the more general available, since results for general $A_\infty$ weights (which include power weights) have been obtained by Haroske and Skrzypczak  in \cite{HS}, where also compactness of the embeddings and entropy numbers are analyzed. However, the proofs in \cite{MV} are much simpler if one is interested in obtaining the embeddings for power weights only. Moreover, they have the advantage of being presented in such a way that they also hold for the general vector-valued case.
 
To  keep our presentation as simple as possible, we will only consider the scalar case, but the vector-valued case can be obtained as in \cite{MV} using Lemma \ref{lemma-frac} and Theorem \ref{desig-npp} below instead of \cite[Lemma 4.5]{MV} and \cite[Proposition 4.1]{MV}, respectively. The same can be done if one is interested in obtaining the improvements in the radial case for embeddings for Triebel-Lizorkin and potential spaces also presented in \cite{MV}.
 
Finally, we remark that until now only unweighted embeddings were known to improve in the radial case with respect to the non-radial case, both in the Besov and Triebel-Lizorkin settings (see, for instance, \cite{SS}). The authors together with N. Saintier are currently working in $A_\infty$ weighted versions in the radial case and their compactness properties \cite{DDS}, but they require completely different, more sophisticated techniques. 

We begin with some necessary definitions.

\begin{definition}[Construction of the Littlewood-Paley partition]\label{def:Phi}
Let $\varphi \in \Schw(\R^n)$ be such that $0\leq \wh{\varphi}(\xi)\leq 1$ for all  $\xi\in \R^n$, 
  $\wh{\varphi}(\xi) = 1$ if $|\xi|\leq 1$, and  $\wh{\varphi}(\xi)=0$ if $|\xi|\geq \frac32$.
 Let $\wh{\varphi}_0 = \wh{\varphi}$, and
$\wh{\varphi}_k(\xi) = \wh{\varphi}(2^{-k}\xi) - \wh{\varphi}(2^{-k+1}\xi)$ 
for all $\xi\in \R^n$ and $k\geq 1$.

Finally, let $\Phi$ be the set of all sequences $(\varphi_n)_{n\geq 0}$ constructed in the above way.
\end{definition}

For  $\varphi$ as in the definition and $f\in \TD(\R^n)$ one sets
\[S_k f := \varphi_k * f = \F^{-1} [\wh{\varphi}_k \wh{f}],\]
which belongs to $C^\infty(\R^n)\cap \TD(\R^n)$. 
Since $\sum_{k\geq 0} \wh{\varphi}_k(\xi) = 1$ for all $\xi\in \R^n$, 
we have $\sum_{k\geq 0} S_kf =  f$ in the sense of distributions.

\begin{definition}\label{def:Besov}
Let $p,q\in [1,\infty]$, $s\in\R$. The (inhomogeneous) Besov space 
$B^s(p,q;\alpha)$ is defined as the space of all
$f\in {\mathcal S}'(\R^n)$ for which 
\[ \|f\|_{B^s (p,q;\alpha)} :=
                                \left(  \sum_{k\ge 0} 2^{ks} \|S_kf\|_{p;\alpha}^q 
                               \right)^{1/q} < \infty. \]
\end{definition}
with the usual modifications for $q = \infty$. The corresponding radial version will be denoted by $B^s_{rad}(p,q;\alpha)$. The following elementary embedding for Besov spaces is well known and will be useful later  to prove the main embedding theorem.

\begin{lemma}
\label{embedding}
For all $q_0, q_1\in [1, \infty], p \in [1, \infty], s \in \R, \gamma > -n$ and $\varepsilon > 0$ there holds
$$
B^{s+\varepsilon}_{rad}(p, q_0, \gamma/p) \hookrightarrow B^s_{rad}(p, q_1, \gamma/p).
$$
\end{lemma}
\begin{proof}
See \cite[Section 2.3.2]{T}. 
\end{proof}

With the above definitions, we are ready to state a weighted Nikol'skij-Plancherel-P\'olya type inequality  for radial functions  (Theorem \ref{desig-npp}) which is the key step in the embedding proof. First we prove the following elementary lemma.

\begin{lemma}
\label{lemma-frac}
If $\eta \in \mathcal{S}(\R^n)$ is a Schwartz function and $g_x(z)=\chi_{B(x,1)}(z)$, there holds $|\eta(x-y)|\le T_\gamma g_x(y)$ for all $0<\gamma<n$.
\end{lemma}
\begin{proof}
Since $\eta$ is a Schwartz function, there is a constant $C$ such that $|\eta(y)|\le C (1+|y|)^{-\gamma}$ for all $y\in \R^n$, whence, noting that if $z\in B(x,1)$ then $|z-y|< 1 + |x-y|$, we have
$$
T_\gamma g_x(y) = \int_{B(x,1)} \frac{1}{|z-y|^\gamma} \, dz \ge  \frac{C}{(1+|x-y|)^\gamma} \ge |\eta(x-y)|.
$$
\end{proof}

\begin{theorem}
\label{desig-npp}
Let $1<p_0, p_1 \le \infty$. Let $\gamma_0,\gamma_1 >-n$. Then, if $f:\mathbb{R}^n\to \mathbb{R}$ is a radial function such that $supp(\widehat f)\subseteq \{ x\in \R^n : |x|<1\}$. Then 
$$
\|f\|_{p_1; \frac{\gamma_1}{p_1}} \le C \|f\|_{p_0; \frac{\gamma_0}{p_0}}
$$
provided that either

\begin{equation}
\label{positivos}
p_0\ge p_1 \quad \mbox{and} \quad \frac{\gamma_0}{p_0} - \frac{\gamma_1}{p_1} > n \left( \frac{1}{p_1}-\frac{1}{p_0}\right)
\end{equation}
or 

\begin{equation}
\label{negativos}
p_0< p_1 \quad \mbox{and} \quad  \frac{\gamma_0}{p_0} - \frac{\gamma_1}{p_1} \ge (n-1)\left( \frac{1}{p_1}-\frac{1}{p_0}\right).
\end{equation}
\end{theorem}

\begin{remark}
It is immediate to see that the above Theorem is an improvement of \cite[Proposition 4.1]{MV} in the radial case, since that Proposition has the hypothesis $\frac{\gamma_0}{p_0} - \frac{\gamma_1}{p_1}\ge 0$ when $p_0<p_1$.
\end{remark}

\begin{proof}\emph{(Theorem \ref{desig-npp})}
It suffices to consider the case $p_0<p_1$ and $\frac{\gamma_0}{p_0}-\frac{\gamma_1}{p_1}<0$ since otherwise the result is contained in \cite[Proposition 4.1]{MV}. We consider separately $p_1=\infty$ and $p_1<\infty$. Moreover, one can check that in order to have $\frac{\gamma_0}{p_0} - \frac{\gamma_1}{p_1}<0$, necessarily $\gamma_1 \ge \gamma_0$ and equality can only hold if they are both negative. We will divide the proof into the following (possibly overlapping) cases:
\begin{enumerate}
\item $p_1=\infty$,
\item $p_0<p_1<\infty$, $0\le \gamma_0 < \gamma_1$,
\item  $p_0<p_1<\infty$, $-n<\gamma_0 < n(p_0-1)$.
\end{enumerate}

\emph{Case 1:} $p_0 < \infty, p_1 = \infty$. Then, under our assumptions, $0>\gamma_0 > -(n-1)$, so that that there exists $\varepsilon \in (0,1)$ such that $\gamma_0 = -(n-1)(1-\varepsilon)$.

Assume first that $f\in L^\infty(\R^n)$ and let $\eta\in \mathcal{S}(\R^n)$ be such that $supp(\widehat \eta)\subseteq B_2$ and $\widehat\eta=1$ on $B_1$. Then one has $f = f * \eta$ and, by Lemma \ref{lemma-frac},

\[
\begin{array}{ll}
|f(x)|  &\le  \int_{\R^n} |f(y)| |\eta(x-y)| \, dy  \\[0.2 cm]
   &\le \int_{\R^n} |f(y)| |T_\gamma g_x(y)| \, dy \\[0.2 cm]
     &\le \int_{\R^n} |f(y)|^{1-p_0/r} |f(y)|^{p_0/r} |T_\gamma g_x(y)| \, dy \\[0.2 cm]
  &\le \|f\|_\infty^{1-p_0/r} \int_{\R^n} |T_\gamma (f^{p_0/r})(y)| |g_x(y)| \, dy
\end{array}
\]

for a parameter $r$, such that $p_0<r<\infty$ to be chosen later. 
Then, using H\"older's inequality, the fact that for any $q\in [1, \infty]$, $\|g_x\|_{q'} < C$ and Theorem \ref{teo-ddd},
\[
\begin{array}{ll}
\|f\|_{L^\infty(\R^n)}  &\le \|f\|_\infty^{1-p_0/r} \|T_\gamma (f^{p_0/r})\|_q \|g_x\|_{q'} \\[0.2 cm]
 &\le  C\|f\|_\infty^{1-p_0/r} \| |f|^{p_0/r} \|_{r; \gamma_0/r} \\[0.2 cm]
 &\le  C\|f\|_\infty^{1-p_0/r} \|f \|_{p_0; \gamma_0/p_0}^{p_0/r}
\end{array}
\]

provided that 

$\frac{1}{q}=\frac{1}{r}+\frac{\gamma_0/r + \gamma}{n}-1$, $q \ge r$, $0<\gamma<n$, $\frac{\gamma_0}{r} < \frac{n}{r'}$, $\frac{\gamma_0}{r} \ge (n-1)(\frac{1}{q}-\frac{1}{r})$. 

At this point, we choose $q$ and $r$ such that $p_0 < r < q$ and $\varepsilon = \frac{r}{q}$, so that $\gamma_0 = (n-1)(\frac{r}{q}-1)$. 

Therefore, $\frac{\gamma_0}{r}=  (n-1)(\frac{1}{q}-\frac{1}{r})$  which proves condition $\frac{\gamma_0}{r} \ge (n-1)(\frac{1}{q}-\frac{1}{r})$.

Condition $\frac{\gamma_0}{r} < \frac{n}{r'}$ follows trivially from the fact that $\gamma_0 <0$. 

From the scaling condition and the fact that $\frac{\gamma_0}{r}= (n-1)(\frac{1}{q}-\frac{1}{r})$, it follows that $n-\gamma = \frac{1}{r}-\frac{1}{q}$. Hence, $n-\gamma >0$ follows from $q>r$ and $n-\gamma<n$ follows from $\frac{1}{r}-\frac{1}{q} = \frac{1}{r}(1-\varepsilon) < n$.

\bigskip 

\emph{Case 2:} $p_0<p_1<\infty$, $0\le \gamma_0 <\gamma_1$. Then, by case 1,
\[ 
\begin{array}{ll} 
\|f\|_{p_1; \gamma_1/p_1} &= \|f\|_{p_0; \gamma_0/p_0}^{p_0/p_1} \||f|^{p_1 -p_0} |x|^{\gamma_1 - \gamma_0}\|_\infty^{1/p_1}\\
& \le C \|f\|_{p_0; \gamma_0/p_0}^{p_0/p_1} \| |f|^{p_1 -p_0} |x|^{\gamma_1 - \gamma_0}\|_{q; \gamma}^{1/qp_1}
\end{array} 
\]
with  $q=\frac{p_0}{p_1-p_0}$ and $\gamma= \gamma_0-\frac{p_0}{p_1-p_0}(\gamma_1 - \gamma_0)$, provided $\gamma \ge -(n-1)$, which follows from the fact that $\frac{\gamma_0}{p_0}- \frac{\gamma_1}{p_1} \ge (n-1) (\frac{1}{p_0}-\frac{1}{p_1})$.
Hence, by the choice of $q$ and $\gamma$,
\begin{equation}
\|f\|_{p_1; \gamma_1/p_1} \le C \|f\|_{p_0; \gamma_0/p_0}^{p_0/p_1} \|f\|_{p_0; \gamma_0/p_0}^{1/p_0-1/p_1} \le C \|f\|_{p_0; \gamma_0/p_0}
\end{equation}
where the last inequality follows from $p_0<p_1$.

\bigskip 

\emph{Case 3:} $p_0 < p_1 < \infty$, $-n <\gamma_0< n (p_0-1)$. By Theorem \ref{rigual1} and Lemma \ref{lema-dualidad} we have
\[ 
\begin{array}{ll} 
\|f\|_{p_1; \gamma_1/p_1} &= \|f* \eta\|_{p_1; \gamma_1/p_1} \\[0.2 cm]
& \le C \| |x|^a \eta\|_\infty \|f\|_{p_0; \gamma_0/p_0} \\[0.2 cm]
&\le C \|f\|_{p_0; \gamma_0/p_0}
\end{array} 
\]
provided that
$\frac{1}{p_1} = \frac{1}{p_0} + \frac{a + \gamma_0/p_0 - \gamma_1/p_1}{n}-1$, $-\gamma_1<n$, $\frac{\gamma_0}{p_0}<\frac{n}{p_0'}$, $\frac{1}{p_1}\le \frac{1}{p_0}$, $0<a<n$, and $\frac{\gamma_0}{p_0}-\frac{\gamma_1}{p_1} \ge (n-1)(\frac{1}{p_1}-\frac{1}{p_0})$. 
All the conditions are trivially satisfied except $0<a<n$ which, because of the scaling condition, is equivalent to $0<\frac{n+\gamma_1}{p_1} - \frac{n+\gamma_0}{p_0} + n < n$.
The RHS of the inequality follows from the fact that, under our assumptions,  $\frac{n+\gamma_0}{p_0}\ge \frac{n+\gamma_1}{p_1} + \frac{1}{p_0} - \frac{1}{p_1} > \frac{n+\gamma_1}{p_1}$. 
The LHS of the inequality follows from the fact that $\frac{n+\gamma_1}{p_1} >0$ and that $n - \frac{n+\gamma_0}{p_0} >0$ because we are assuming $\gamma_0 < n(p_0-1)$.

\end{proof}

\begin{corollary}
\label{teo-derivadas}
Let $1<p_0, p_1 \le \infty$. Let $\gamma_0,\gamma_1 >-n$. Then, there exists $C>0$ such that for every radial function $f:\mathbb{R}^n\to \mathbb{R}$ satisfying $supp(\widehat f)\subseteq B_t$ for some $t>0$, there holds 
$$
\| f\|_{p_1; \gamma_1/p_1} \le C t^{\delta} \|f\|_{p_0; \gamma_0/p_0}  
$$
provided  (\ref{positivos}) or (\ref{negativos}) hold, and
\begin{equation}
\label{def-delta}
\delta:=\frac{n+\gamma_0}{p_0}-\frac{n+\gamma_1}{p_1}.
\end{equation}
\end{corollary}

\begin{proof} Note that if $supp\,\hat f\subset B_t$ then $f_t(x):=t^{-n}f(x/t)$ satisfies $\hat f_t(\xi)=\hat f(t\xi)$ and $supp\,\hat f_t\subset B_1$. Therefore, 
by Theorem \ref{desig-npp},
 $$ \|f_t\|_{p_1; \gamma_1/p_1} \le C \|f_t\|_{p_0; \gamma_0/p_0} $$ 
that is, 
$$ \|f\|_{p_1; \gamma_1/p_1} \le Ct^\delta \|f\|_{p_0; \gamma_0/p_0} $$ 
\end{proof}

\begin{corollary}
\label{coro}
Let $1<p_0, p_1 \le \infty$ and $\gamma_0,\gamma_1 >-n$ . 
Assume that
$$\delta \ge \max\left\{0 , \frac{1}{p_0}-\frac{1}{p_1}\right\}$$
and $\delta \neq 0$, where $\delta$ is defined by (\ref{def-delta}).
Then 
$$B^{s+\delta}_{rad}(p_0, q, \gamma_0/p_0)  \hookrightarrow B^s_{rad}(p_1, q, \gamma_1/p_1).$$
\end{corollary}

\begin{proof} 
Notice that since $\varphi$ is taken to be radial, then the $\varphi_k$ are also radial.

Let $f \in B^{s+\delta}_{rad}(p_0, q, \gamma_0/p_0)$. By density, we can assume that $f \in\mathcal{S}$ and is radial,  whence 
 $\widehat{f*\varphi_k}=\hat f \hat\varphi_k$ and $f*\varphi_k$ are also radial. 
 
Since $supp\,\widehat{f*\varphi_k} \subset supp\,\hat\varphi_k\subset B(C2^k)$, by Corollary \ref{teo-derivadas} we have
$$ \|f*\varphi_k\|_{p_1; \gamma_1/p_1} \le C2^{k\delta} \|f*\varphi_k\|_{p_0; \gamma_0/p_0} $$ 
and, therefore,  
\[
\begin{array}{ll}
 \|f\|^q_{B^s(p_1,q, \gamma_1/p_1)} &=\sum_{k\ge 0} 2^{qks}\|f*\varphi_k\|^q_{p_1; \gamma_1/p_1} \\[0.2 cm]
  & \le C\sum_{k \ge 0} 2^{qk(s+\delta)}\|f*\varphi_k\|^q_{p_1; \gamma_1/p_1} \\[0.2 cm]
  & = C\|f\|^q_{B^{s+\delta}(p_1,q, \gamma_1/p_1)} 
  \end{array}
   \]

\end{proof}

\begin{theorem} 
\label{teo-inmersion}
Let $1<p_0, p_1 \le \infty$, $s_0,s_1\in\mathbb{R}$, and $\gamma_0,\gamma_1 >-n$. Then 
$$
B^{s_0}_{rad}(p_0, q_0, \gamma_0/p_0) \hookrightarrow B^{s_1}_{rad}(p_1, q_1, \gamma_1/p_1)
$$
provided  that

$$  \delta\ge \max \left\{0,\frac{1}{p_0}-\frac{1}{p_1}\right\}, $$ 
and  
\begin{itemize}
\item either $s_0-s_1> \delta$
\item or $s_0-s_1= \delta$ and $q_1\ge q_0$
\end{itemize} 
where $\delta$ is as in (\ref{def-delta}) and the case $\delta=0$ is only admissible if $\gamma_0=\gamma_1$ and $p_0=p_1$.
\end{theorem}

\begin{proof}
In the case $\delta=0, \gamma_0=\gamma_1, p_0=p_1$, the result is true even in the non-radial case, and is contained in \cite[Theorem 1.1]{MV}.

If $s_0 - s_1= \delta$ and $q_1 \ge q_0$, the theorem reduces to Corollary \ref{coro}.

If $s_0 - s_1 > \delta$, let  $\varepsilon>0$ satisfy $s_0 - s_1  = \delta + \varepsilon$. By the previous case and Lemma \ref{embedding}, we obtain the chain of continuous embeddings
$$
B^{s_0+\varepsilon}_{rad}(p_0, q_0, \gamma_0/p_0) \hookrightarrow B^{s_0}_{rad}(p_0, q_1, \gamma_0/p_0) \hookrightarrow B^{s_1}_{rad}(p_1, q_1, \gamma_1/p_1)
$$ 
which concludes the proof. 

\end{proof}

\bigskip
\textbf{Acknowledgements}
The authors thank Carlos D'Andrea for reference \cite{B}. We are also indebted to the anonymous referee for helpful comments and suggestions.

\end{document}